\documentclass{amsart}[11pt]

\def\volume{\operatorname{vol}}
\def\vol{\operatorname{vol}}
\def\op{\operatorname}

\usepackage[usenames,dvipsnames]{color}
\usepackage{amsmath,amssymb,bm,amsthm, amsfonts, mathrsfs, eucal, epsfig}
\usepackage{graphicx}
\usepackage{url}
\usepackage[top=2.5cm,bottom=2.0cm,left=2.5cm,right=2.5cm]{geometry}

\makeatletter 
\@addtoreset{equation}{section}
\makeatother  

\begin{document}

\newtheorem{Thm}{Theorem}[section]
\newtheorem{Def}[Thm]{Definition}
\newtheorem{Lem}[Thm]{Lemma}
\newtheorem{Rem}[Thm]{Remark}
\newtheorem{Prop}[Thm]{Proposition}
\newtheorem{Cor}[Thm]{Corollary}
\newtheorem{sublemma}{Sub-Lemma}

\newtheorem{Example}{Example}[section]
\newcommand{\g}[0]{\textmd{g}}
\newcommand{\pr}[0]{\partial_r}
\newcommand{\dif}{\mathrm{d}}
\newcommand{\bg}{\bar{\gamma}}
\newcommand{\md}{\rm{md}}
\newcommand{\cn}{\rm{cn}}
\newcommand{\sn}{\rm{sn}}
\newcommand{\seg}{\mathrm{seg}}
\newcommand{\dd}{\,\mathrm d}
\newcommand{\Ric}{\mbox{Ric}}
\newcommand{\Iso}{\mbox{Iso}}
\newcommand{\ra}{\rightarrow}
\newcommand{\Hess}{\mathrm{Hess}}
\newcommand{\R}{\mathbb{R}}
\newcommand{\N}{\mathbb{N}}
\newcommand{\RCD}{\operatorname{RCD}}
\newcommand{\GH}{\mathrm{GH}}
\newcommand{\pmGH}{\mathrm{pmGH}}
\newcommand{\Span}{\operatorname{span}}
\newcommand{\esssup}{\operatorname*{ess\,sup}}
\newcommand{\HH}{\mathcal{H}}

\newcommand{\ip}[2]{\left\langle #1,#2\right\rangle}

\title[Transformation theorem]{Transformation theorem and monotonicity of the numbers of almost Euclidean factors of geodesic balls}
\author{Lina Chen}
\address[Lina Chen]{School of Mathematics and Statistics, Nanjing University of Science and Technology, Nanjing China}

\email{chenlina\_mail@163.com}
\thanks{Supported partially by NSFC Grant 12471046, Jiangsu NSF Grant BK20240185 and the Fundamental Research Funds for the Central Universities No. 30923010201.} 

\maketitle

\begin{abstract}

\setlength{\parindent}{10pt} \setlength{\parskip}{1.5ex plus 0.5ex
minus 0.2ex} 

In \cite{HH}, H. Huang-X.Huang introduced the generalized Reifenberg condition which describes the non-increasing property of numbers of almost Euclidean factors of geodesic balls and gave a transformation theorem under this condition. In this note, we study how almost Euclidean factors of concentric geodesic balls interact
with harmonic splitting maps under a lower Ricci curvature bound and prove
a dimension gap for harmonic functions of almost linear growth 
and then derive a transformation theorem under a non-decreasing property compared with the non-increasing property above.
 We also construct complete
manifolds with nonnegative Ricci curvature whose balls are uniformly almost
$1$-Euclidean at every scale, but for which no single harmonic function can
be transformed into a splitting function at all smaller scales which give a negative answer to the existence question for all-scale
splitting functions in \cite{HH}. 

 \end{abstract}

\section{Introduction}

The transformation technique for almost splitting maps originates from Cheeger-Naber (\cite{ChN}) where it plays a key role in proving the codimension 4 conjecture. And in \cite{CJN}, Cheeger-Jiang-Naber gave a geometric transformation theorem which says that under some geometric conditions, if a manifold $(M, p)$ with almost nonnegative Ricci curvature has uniform volume lower bound of a unit geodesic ball $B_1(p)$, i.e., non-collapsing,  then a $(\delta, k)$-splitting map of $B_1(p)$ is still a $(\epsilon, k)$-splitting map in smaller scales $B_r(p)$, $r\ll 1$, up to a transformation. For $\op{RCD}$-spaces, geometric transformation theorem with non-collapsing condition also holds (\cite{BNS}). 
In \cite{HH}, H. Huang-X. Huang gave a transformation theorem in collapsing case under a $(\Phi, r_0; k,\delta)$-generalized Reifenberg condition. And in \cite{HH}, they asked whether that $B_r(p)$ is $(\delta, k)$-Euclidean for $r\in (0,1]$  where $\delta$ is sufficient small could guarantee transformation theorem. In this note, we will provide another sufficient condition for transformation theorem (Theorem~\ref{trans}) and give a negative answer to their question (Theorem~\ref{main-2}). 

Recall that in \cite[Definition 1.4]{HH}, for an $n$-manifold $M$ and $p\in M$, one says that $p$ satisfies $(\Phi, r_0; k,\delta)$-generalized Reifenberg condition if there exist $\delta>0$, $k\in \Bbb Z^+$, $r_0>0$ and a function $\Phi: \Bbb R^+\to \Bbb R^+$ with $\lim_{\delta'\to 0}\Phi(\delta')=0$ such that for $0<r<r_0$, 
\begin{equation}\Theta_{p, k'}(r)\leq \Phi(\max\{\delta, \theta_{p, k'}(r)\}), \,\forall\, k'\geq k, \label{non-inc}\end{equation}
where
\begin{equation*}\theta_{p, k}(r)=\inf\{\epsilon>0 | B_r(p) \text{ is } (\epsilon, k)\text{-Euclidean}\}, \end{equation*}
$$\Theta_{p, k}(r)=\sup_{s\in [0, r]}\theta_{p, k}(s).$$ 
And we say $B_r(p)$ is $(\delta, k)$-Euclidean if there is a metric space $(Z, d_Z)$ such that
$$d_{GH}(B_r(p), B_r((0, x)))\leq \delta r,$$
where $(0, x)\in \Bbb R^k\times Z$. This generalized Reifenberg  condition \eqref{non-inc} implies that the number of almost Euclidean factors of $B_r(p)$ in the Gromov-Hausdorff sense is non-increasing for $r\in (0, r_0)$. Similarly, we define the non-decreasing property as following:
for any $r\in (r_0, r_1)$ where $r_1$ may be infinity, let
$$\hat \Theta_{p, k}(r)=\sup_{s\in [r, r_1)}\theta_{p, k}(s),$$
and assume that
\begin{equation}\hat\Theta_{p, k'}(r)\leq \Phi(\max\{\delta, \theta_{p, k'}(r)\}), \,\forall\, k'\geq k. \label{non-dec}\end{equation}

The first result about transformation theorem is 
\begin{Thm} \label{trans}
Given $n, \epsilon>0$, integer $n$, and a function $\Phi: \Bbb R^+\to \Bbb R^+$ with $\lim_{\delta'\to 0}\Phi(\delta')=0$, there are $\delta_0=\delta(n, \epsilon, \Phi)>0, R_0=R(n, \epsilon, \Phi)>4$ such that for $1\leq k\leq n$, $0<\delta<\delta_0$, $R\geq R_0$ if a complete $n$-manifold with $\op{Ric}_M\geq -(n-1)\delta$, satisfies \eqref{non-dec} and $\theta_{p, k}(r)<\delta$ for
 some $p\in M$  and $0<a\leq r\leq R$, $a\leq 1$, then there exists a harmonic $(\epsilon,k)$-splitting map
$v:B_2(p)\to\R^k$ with $v(p)=0$ and $G_2(v)=I_k$ such that,
for every $a\leq s\leq1$, its lower triangular energy normalization
$T_s$, with positive diagonal, makes $T_sv$ an
$(\epsilon,k)$-splitting map on $B_s(p)$, with $G_s(T_sv)=I_k$.
Moreover,
\begin{equation}\label{eq:target-matrices}
 \|T_sT_t^{-1}\|\leq(1+C(n)\epsilon)
 \max\{(s/t)^{C(n)\epsilon},(t/s)^{C(n)\epsilon}\},
 \qquad s,t\in[a,1].
\end{equation}where
\begin{equation*}\label{eq:gram-matrix}
 G_r(v)_{ab}=\frac{1}{\volume(B_r(p))}\int_{B_r(p)}\ip{\nabla v^a}{\nabla v^b} d \volume.
\end{equation*} 
\end{Thm}

Theorem~\ref{trans} and \cite[Theorem 1.8]{HH} implies that if the number of almost Euclidean factors of $B_r(p)$ in the Gromov-Hausdroff sense is monotone respect to $r$, then transformation theorem holds. And examples can be constructed by using \cite{CN}'s results which shows that if the monotonic property does not hold, transformation theorem would be false.

\begin{Thm}\label{main-2}
For every $n\geq 4$, there exists $\epsilon_*(n)>0$ such that, for every
$0<\delta<1$, there is a complete smooth pointed $n$-manifold $(M,g,q)$
with $\op{Ric}_M\geq0$ satisfying the following properties:
\begin{enumerate}
\item $B_s(q)$ is $(\delta,1)$-Euclidean for every $0<s\leq2$;
\item $B_2(q)$ admits a $(\delta,1)$-splitting function;
\item there is no harmonic function $u:B_2(q)\to\R$ such that for every
$0<s\leq1$ some $a_s>0$ makes $a_su$ an
$(\epsilon_*,1)$-splitting function on $B_s(q)$.
\end{enumerate}
\end{Thm}

To derive Theorem~\ref{trans}, we will use a method as in \cite{CJN,HH}. In particular, comparing with \cite[Theorem 1.10]{HH}, we have the following dimension gap result for almost linear harmonic functions. Let \begin{equation*}
 \HH_\gamma(X,p):=\{f:\;\Delta f=0\text{ on }X,\quad f(p)=0,
 |f(x)|\leq C_f(1+d(x,p)^\gamma)\text{ for every }x\in X\}.
\end{equation*} 
\begin{Thm}[Dimension gap]
\label{thm:gap}
Given $N\in[1,\infty)$, $\eta>0$, and an integer $h\geq1$, there exists
 $\alpha_*=\alpha_*(N,\eta,h)>0$ such that the following holds.
If $(X,p,d,m)$ is an $\op{RCD}(0,N)$-space satisfying
\begin{equation}
\label{eq:non-splitting}
 B_r(p)\text{ is not }(\eta,h)\text{-Euclidean}
 \qquad\text{for every }r\geq1,
\end{equation}
then
\begin{equation}
\label{eq:gap-conclusion}
 \dim\HH_{1+\alpha}(X,p)\leq h-1
 \qquad\text{whenever }0<\alpha<\alpha_*.
\end{equation}
\end{Thm}
By this dimension gap theorem, we will also provide a connection of almost Euclidean factors of different concentric geodesic balls of a manifolds and its universal cover (see Theorem~\ref{thm:partial-transfer} and compare with \cite[Lemma 3.7]{H}). 

The paper is organized as follows. In section 2, we will give some preliminaries. And in section 3 the proof of Theorem~\ref{thm:gap} will be presented. We will give the results of Theorem~\ref{trans} and other applications of Theorem~\ref{thm:gap} in section 4. In section 5, we will construct examples of Theorem~\ref{main-2}.

\section{Preliminaries}

In this section, we will supply some definitions and properties that we need in the following. Most of the settings are just as in \cite{HH}  and we recommend one to see the Section 2 there for more details.   We use $d_{\mathrm{pGH}}$ and $d_{\mathrm{pmGH}}$ for distances metrizing pointed and pointed measured Gromov--Hausdorff convergence on bounded balls.
Unqualified $d_{\GH}$ between compact balls is the ordinary Gromov-Hausdorff distance. In manifolds with lower Ricci curvature bound or $\op{RCD}$-spaces, a very important character is the $(\delta, k)$-splitting map which originates from \cite{CC,CC0} and introduced in \cite{ChN} for manifolds and Ricci limit spaces and in \cite{BPS} for $\op{RCD}$-spaces.

\begin{Def}
For a $\op{RCD}(-(N-1)\delta, N)$-space $(X, p, d, m)$, a harmonic map $u=(u^1, \cdots, u^k): B_r(p)\to \Bbb R^k$ is a $(\delta, k)$-splitting map if 
\begin{enumerate}
\item[(1)]  {$|\nabla u^a|\leq 1+\delta$;}
\item[(2)] {$\frac{r^2}{m(B_r(p))}\int_{B_r(p)} |\op{Hess}_{u^a}|^2dm =r^2-\kern-1em\int_{B_r(p)}|\op{Hess}_{u^a}|^2 dm<\delta
$;}
\item[(3)] {$-\kern-1em\int_{B_r(p)}|\left<\nabla u^a, \nabla u^b\right>-\delta^{ab}| dm<\delta$} for any $a, b=\{1, \cdots, k\}$.
\end{enumerate}
\end{Def}
Note that by Bochner inequality and good cut-off function, (2) in smaller ball can be derived by (1) and (3). A  $(\delta, k)$-splitting map is closed related with the almost splitting structure (cf. \cite{CC, Gi, BNS}):
\begin{Thm}	\label{split-map}
Given $\delta \ll 1, N>0$, there exists $\epsilon(N, \delta) > 0$ such that for any $\epsilon \in (0, \epsilon(N, \delta))$ for any $(X, p, d, m) \in \mathsf{RCD}(-\epsilon(N - 1), N)$,
\begin{enumerate}
    \item[\textbf{(1):}]  if for some integer $k$ and for a metric measure space $(Z, z, d_Z, m_Z)$,
    \begin{equation}d_{\mathrm{mGH}}(B_4(p), B_4^{\mathbb{R}^k \times Z}(0^k, z)) \leq \epsilon \label{mgh}\end{equation} 
then there exists a \((\delta, k)\)-splitting map \(u = (u^1, \dots, u^k) : B_1(p) \to \mathbb{R}^k\);
    \item[\textbf{(2):}] if there exists an \((\epsilon, k)\)-splitting map \(u : B_4(p) \to \mathbb{R}^k\) for some integer \(k\), then \(d_{\mathrm{mGH}}(B_1(p), B_1^{\mathbb{R}^k \times Z}(0^k, z)) \leq \delta\) for some pointed metric measure space \((Z, z, d_Z, m_Z)\).
\end{enumerate}
\end{Thm}
One can see \cite[Remark 2.10]{HH} for the difference if one changes \eqref{mgh} to $d_{GH}(B_4(p), B_4^{\mathbb{R}^k \times Z}(0^k, z)) \leq \epsilon$.  And in the proof of Theorem~\ref{trans}, we need the following property of $(\delta, k)$-splitting map which were known to experts:
\begin{Lem}\label{lem:uniform-selection}
Given $n\geq2$, $K\geq1$, and $\xi>0$, there is
$\sigma=\sigma(n,K,\xi)\in(0,1)$ with the following property.
Let $(M^n,g,p)$ be complete, $\Ric_g\geq-(n-1)\kappa g$ with
$\kappa\geq0$ and $\kappa R^2\leq\sigma$, and let
$z:B_R(p)\to\R^m$, $m\leq n$, be harmonic. Suppose
\[
 \sup_{B_R(p)}|\nabla z^a|\leq K,\qquad
 -\kern-1em\int_{B_R(p)}
 |\ip{\nabla z^a}{\nabla z^b}-\delta_{ab}|\leq\sigma
 \quad(1\leq a,b\leq m).
\]
Then, simultaneously for all $R/8\leq r\leq R/2$ and
$1\leq h\leq m$:
\begin{enumerate}
\item If $O\in\R^{h\times m}$ has orthonormal rows, then $Oz$ is
$(\xi,h)$-splitting on $B_r(p)$.
\item For every constant rank-$h$ matrix $B\in\R^{h\times m}$,
$G_r(Bz)$ is positive definite, and its unique lower triangular
energy normalization $T_rBz$ is $(\xi,h)$-splitting on $B_r(p)$.
The threshold is independent of the norm and condition number of $B$.
\end{enumerate}
\end{Lem}
\begin{proof}
Rescale $M$ such that $R=1$ with $g'=R^{-1}g$ and $z'=R^{-1}z$, and write
$C=C(n,K)$. If $OO^{\mathsf T}=I_h$, put $w=Oz$.
Expanding the Gram entries and using the unit lengths of the rows gives
\[
 -\kern-1em\int_{B_1}
 |\ip{\nabla w^a}{\nabla w^b}-\delta_{ab}|
 \leq n\sigma,
 \qquad \sup_{B_1(p)}|\nabla w^a|\leq\sqrt n K.
\]
For $F=|\nabla w^a|^2$, Bochner's inequality gives
$\Delta(F-1)_+\geq-C\sigma$. The standard local Moser estimate
therefore bounds $\sup_{B_{1/2}}(F-1)_+$ by $C\sigma$.
Integrating Bochner's inequality against a Laplacian cutoff equal
to one on $B_{1/2}$ and supported in $B_1$, and subtracting $1$
in the term $\int F\Delta\chi$, gives the Hessian estimate below;
see \cite[Lemma 3.1]{MN} for the cutoff. Volume comparison then yields,
uniformly for $1/8\leq r\leq1/2$,
$$ \sup_{B_{1/2}}|\nabla w^a| \leq1+C\sigma,$$
 $$r^2-\kern-1em\int_{B_r}|\Hess w^a|^2\leq C\sigma,$$
$$ -\kern-1em\int_{B_r}
 |\ip{\nabla w^a}{\nabla w^b}-\delta_{ab}|\leq C\sigma.$$
These estimates prove (1) for sufficiently small $\sigma$.

For a rank-$h$ matrix $B$, take the Cholesky decomposition
\[
 BB^{\mathsf T}=LL^{\mathsf T},\qquad O=L^{-1}B,
 \qquad OO^{\mathsf T}=I_h.
\]
Thus $Bz=Lw$, where $w=Oz$ satisfies the estimates above, and
$H_r:=G_r(w)=I_h+O_n(\sigma)$ is uniformly positive definite.
Its lower triangular energy normalization satisfies
$C_r=I_h+O_n(\sigma)$ by smooth dependence of Cholesky decomposition.
Consequently $C_rw$ still satisfies all three splitting estimates
with error $C\sigma$ on $B_r$.
Since
\[
 T_r:=C_rL^{-1},\qquad T_rBz=C_rw,\qquad G_r(T_rBz)=I_h,
\]
$T_r$ is lower triangular with positive diagonal and hence is the
canonical energy normalization of $Bz$. Only $O$ enters the analytic
estimates; no bound on $L$ or $L^{-1}$ is needed. Choosing
$\sigma$ so that $C\sigma<\xi$ proves (2), uniformly in $B$ and
$r$, and scaling back gives the stated radius interval.
\end{proof}

In \cite{HH}, Huang-Huang derived the following transformation theorem. And in our proof of Theorem~\ref{trans}, a base step is from the case as their assumption. 

\begin{Thm}[\cite{HH}]\label{thm:fixedrank}
Given $N\geq1$, $\epsilon>0$, and $\eta>0$, there is
$\delta_0(N,\epsilon,\eta)>0$ such that the following holds.
Let $(X,d,\mu)$ be $\RCD(-(N-1)\delta,N)$, $0<\delta<\delta_0$,
and $0<a<1$. Suppose that every $B_s(p)$, $a\leq s\leq1$, is
$(\delta,k)$-Euclidean and is not $(\eta,k+1)$-Euclidean.
If $u:B_1(p)\to\R^k$ is $(\delta,k)$-splitting, then for each
$s\in[a,1]$ its energy normalization $T_su$ is
$(\epsilon,k)$-splitting. 
\end{Thm}

And we will also use the linear-growth blowdown principle in \cite[proposition 3.5]{HH}.
\begin{Thm}[\cite{HH}]
For a complete non-compact $\op{RCD}(0, N)$-space $(X, d, m)$, $p\in X$,  if the dimension $\op{dim}\mathcal H_1(X, p)=k$ then any tangent cone at infinity of $X$ is $k$-splitting.
\end{Thm}

\section{Dimension gap theorem}
In this section, we will give the result of dimensional gap result. 
To prove Theorem~\ref{thm:gap}, we first give a uniform normalization of functions in $\HH_{1+\alpha}(X, p)$. Recall that for $\gamma\geq 1$,
\begin{equation*}
 \HH_\gamma(X,p):=\{f:\;\Delta f=0\text{ on }X,\quad f(p)=0,
 |f(x)|\leq C_f(1+d(x,p)^\gamma)\text{ for every }x\in X\}.
\end{equation*}
\begin{Lem}
\label{lem:normalization}
Let $F$ be an $h$-dimensional real vector space of global harmonic functions
 on $(X,p,d,m)\in \RCD(0, N)$, where $h\geq1$, and assume $f(p)=0$ for every $f\in F$.
Suppose $0<\alpha\leq1$ and, for each $f\in F$, there is a finite constant
 $C_f$ such that
\begin{equation}
\label{eq:input-growth}
 P_r(f):=\esssup_{B_r(p)}|\nabla f|\leq C_f r^\alpha,
 \qquad r\geq1.
\end{equation}
Then there are a scale $\rho\geq1$ and a basis $f^1,\ldots,f^h$ of $F$
 with the following property. On
\begin{equation}
\label{eq:rescaling}
 (\bar X,\bar p,\bar d,\bar m)
 :=\left(X,p,\rho^{-1}d,\frac{m}{m(B_\rho(p))}\right),
 \qquad u^a:=\rho^{-1}f^a,
\end{equation}
we have $u^a(\bar p)=0$, $\bar\Delta u^a=0$, and
\begin{equation}
 -\kern-1em\int_{B_1(\bar p)}
 \ip{\bar\nabla u^a}{\bar\nabla u^b}\dd\bar m
 =\delta_{ab},
 \label{eq:normalized-energy} \end{equation}
 \begin{equation}
 \esssup_{B_t(\bar p)}|\bar\nabla u^a|
 \leq C(N,h)t^{2h\alpha},\qquad t\geq1,
 \label{eq:normalized-gradient}\end{equation}
 \begin{equation}
 |u^a(x)|
 \leq C(N,h)\bigl(1+\bar d(x,\bar p)^{1+2h\alpha}\bigr),
 \qquad x\in\bar X.
 \label{eq:normalized-global}
\end{equation}
\end{Lem}

\begin{proof}

Note that the kernel $\ker P_r$ form a decreasing family of subspaces of $F$.
Moreover,
\[
 \bigcap_{r\geq1}\ker P_r=\{0\}.
\]
Since $F$ is
 finite-dimensional, there is
 $r_0\geq1$ such that for $r\geq r_0$, $\ker P_r=\{0\}$ and $P_r$ is a norm. 

Fix a reference basis $e^1,\ldots,e^h$ of $F$. For $r\geq r_0$, let
\[
 K_r:=\{\ell\in F^*: |\ell(f)|\leq P_r(f)\text{ for every }f\in F\}
\]
be the dual unit ball. These are compact convex sets, and
 $K_r\subset K_s$ whenever $r\leq s$. Define
\begin{equation}
\label{eq:determinant}
 D(r):=\max_{\ell_1,\ldots,\ell_h\in K_r}
       \bigl|\det(\ell_a(e^b))_{a,b=1}^h\bigr|.
\end{equation}
Then since $P_r$ is a norm, we have that $D_r$ is positive and finite. In fact, since $F$ has finite dimension and $P_r$ is a norm, there is $C_r>0$ such that for $f=\sum c_be^b$, 
$$P_r(f)\geq C_r(\sum |c_b|^2)^{1/2}.$$
For the coordinate dual basis $e_{*}^1, \cdots, e_{*}^h$, set $l_a=C_r e_{*}^a$, then $l_a\in K_r$ and $(l_a(e^b))=C_rI_h$. And thus $D(r)>0$.

Applied \eqref{eq:input-growth}
 to the reference basis, there is a finite $A$, which may depend on $F$,
 such that
\begin{equation}
\label{eq:D-growth}
 D(r)\leq A r^{h\alpha},\qquad r\geq r_0.
\end{equation}
In particular,
\[
 \frac{D(r)}{r^{2h\alpha}}\to0\qquad(r\to\infty).
\]
Choose $R\geq r_0$ satisfying
\[
 \frac{D(R)}{R^{2h\alpha}}
 \geq\frac12\sup_{r\geq r_0}\frac{D(r)}{r^{2h\alpha}}.
\]
Then, simultaneously for all $t\geq1$,
\begin{equation}
\label{eq:D-doubling}
 \frac{D(tR)}{D(R)}\leq2t^{2h\alpha}.
\end{equation}

Choose $\ell_1,\ldots,\ell_h\in K_R$ realizing $D(R)$. They form a basis
 of $F^*$. For $\ell\in K_{tR}$, write
\[
 \ell=\sum_{j=1}^h a_j\ell_j.
\]
Then
$$\op{det}(\ell_1, \cdots,\ell_{j-1}, \ell, \ell_{j+1}, \cdots, \ell_h)=a_j\op{det}(\ell_1, \cdots, \ell_h).$$
Then $P_R\leq P_{tR}$ and $K_R\subset K_{tR}$ give
\[
 |a_j|\leq\frac{D(tR)}{D(R)},\qquad 1\leq j\leq h.
\]
Hence, for every $f\in F$,
\[
 |\ell(f)|
 \leq h\frac{D(tR)}{D(R)}P_R(f).
\]
Taking the supremum over $\ell\in K_{tR}$ and using the dual description
 of a finite-dimensional norm and Hahn-Banach theorem yield
\begin{equation}
\label{eq:whole-space-control}
 P_{tR}(f)\leq2h\,t^{2h\alpha}P_R(f)
 \quad\text{for every }f\in F\text{ and }t\geq1.
\end{equation}

Set $\rho=2R$ and define
\[
 Q_\rho(f,g):=
 -\kern-1em\int_{B_\rho(p)}\ip{\nabla f}{\nabla g}\dd m.
\]

By the local gradient estimate of \cite[Theorem 1.1]{Ji}, combined with the Poincar\'e inequality and Bishop–Gromov volume comparison, we obtain
\begin{equation}
\label{eq:P-Q}
 P_R(f)^2\leq C(N)Q_\rho(f,f).
\end{equation}
Since $\op{ker}P_R=\{0\}$, $Q_\rho$ is positive definite on $F$.
Choose a $Q_\rho$-orthonormal basis $f^1,\ldots,f^h$.
Then $P_R(f^a)\leq C(N)$. For $t\geq1$, \eqref{eq:whole-space-control}
 gives
\begin{equation*}
 P_{\rho t}(f^a)
 =P_{2Rt}(f^a)
 \leq2h(2t)^{2h\alpha}P_R(f^a)
 \leq C(N,h)t^{2h\alpha}.
\end{equation*}

Under the simultaneous distance and function rescaling
 \eqref{eq:rescaling}, we have
\[
 |\bar\nabla(\rho^{-1}f)|=|\nabla f|,
 \qquad B_t(\bar p)=B_{\rho t}(p).
\]
This proves \eqref{eq:normalized-energy} and
 \eqref{eq:normalized-gradient}. Finally, use the locally Lipschitz
 representative and integrate the local Lipschitz bounds along a minimizing
 geodesic from $\bar p$ to $x$. Taking a ball of radius
 $\max\{1,2\bar d(x,\bar p)\}$, and using $u^a(\bar p)=0$, gives
 \eqref{eq:normalized-global}. All additional factors are bounded in terms
 of $h$, because $0<\alpha\leq1$.
\end{proof}

\begin{proof}[Proof of Theorem~\ref{thm:gap}]
Suppose no such $\alpha_*$ exists. There are $\alpha_i\to 0$ and
 spaces $(X_i,p_i,d_i,m_i)$ satisfying \eqref{eq:non-splitting}, together
 with $h$-dimensional subspaces
\[
 F_i\subset\HH_{1+\alpha_i}(X_i,p_i).
\]
Discard finitely many terms so that $\alpha_i\leq1$.
Note that for $0<\alpha\leq1$,
 $f\in\HH_{1+\alpha}(X,p)$ implies that
 $$\esssup_{B_r(p)}|\nabla f|\leq C_f' r^\alpha
 \quad (r\geq1).$$
 By Lemma~\ref{lem:normalization}, there are $\rho_i\geq1$
 and bases  $u_i^1,\ldots,u_i^h$ on
\[
 (\bar X_i,\bar p_i,\bar d_i,\bar m_i)
 :=\left(X_i,p_i,\rho_i^{-1}d_i,
              \frac{m_i}{m_i(B_{\rho_i}(p_i))}\right)
\]
which are globally harmonic, vanish at $\bar p_i$, and satisfy
\begin{equation} -\kern-1em\int_{B_1(\bar p_i)}
 \ip{\bar\nabla u_i^a}{\bar\nabla u_i^b}\dd\bar m_i
 =\delta_{ab},
 \label{eq:sequence-energy}
 \end{equation}
\begin{equation} |u_i^a(x)|
 \leq C(N,h)\bigl(1+
       \bar d_i(x,\bar p_i)^{1+2h\alpha_i}\bigr).
 \label{eq:sequence-growth}
\end{equation}
The corresponding gradient bound
 \eqref{eq:normalized-gradient} also holds uniformly in $i$.
 
Passing to a subsequence, assume that
\[
 (\bar X_i,\bar p_i,\bar d_i,\bar m_i)
 \xrightarrow{\pmGH}(X_\infty,p_\infty,d_\infty,m_\infty).
\]
On each fixed ball, \eqref{eq:sequence-growth} and the uniform gradient
 bound give uniform function and energy bounds. 
 And by \cite[Theorem~4.4 and Corollary~4.5]{AH} and a diagonal
 argument , $u_i^a\to u_\infty^a$ locally uniformly and strongly in $W^{1,2}_{\mathrm{loc}}$.
Each $u_\infty^a$ is globally harmonic and vanishes at $p_\infty$.
For a fixed $x\in X_\infty$, take points $x_i\to x$ in the
 convergence realization. Passing to the limit,
 \eqref{eq:sequence-growth} gives
\begin{equation}
\label{eq:limit-linear}
 |u_\infty^a(x)|\leq C(N,h)(1+d_\infty(x,p_\infty)),
 \qquad x\in X_\infty.
\end{equation}

Strong convergence of the gradient energies, applied also to
 $u_i^a+u_i^b$ and $u_i^a-u_i^b$, preserves the energy matrix:
\begin{equation}
\label{eq:limit-energy}
 -\kern-1em\int_{B_1(p_\infty)}
 \ip{\nabla u_\infty^a}{\nabla u_\infty^b}\dd m_\infty
 =\delta_{ab}.
\end{equation}
For every $c=(c_1,\ldots,c_h)\in\R^h$, \eqref{eq:limit-energy} implies
\[
 -\kern-1em\int_{B_1(p_\infty)}
 \left|\nabla\sum_{a=1}^h c_a u_\infty^a\right|^2\dd m_\infty
 =\sum_{a=1}^h c_a^2.
\]
Hence the limit functions are linearly independent, and
\begin{equation}
\label{eq:dimension-limit}
 \dim\HH_1(X_\infty,p_\infty)\geq h.
\end{equation}

By \cite[Proposition 3.5]{HH}, a tangent cone at infinity of $X_\infty$
 splits $\R^h$. Consequently, for some sufficiently large but finite
 $L\geq1$, the ball $B_L(p_\infty)$ is $(\eta/4,h)$-Euclidean.
Fix this $L$. Pointed Gromov--Hausdorff convergence then makes
 $B_L(\bar p_i)$ $(\eta/2,h)$-Euclidean for all sufficiently large $i$.
This contradicts the rescaled non-splitting condition and proves the theorem.
\end{proof}

A direct corollary of Theorem~\ref{thm:gap} gives that 
 \begin{Prop} \label{tech}
 Given $\eta>0$, and integers $N\geq k, l, k+l\geq 0$, there is $\epsilon(N, \eta)>0$ such that for $\epsilon<\epsilon(N, \eta)$, for a  $k$-splitting $\op{RCD}(0, N)$-space $(X, p, d, m)$ that $B_r(p)$ is not $(\eta, k+l+1)$-Euclidean for any $r\geq 1$ and for $k+l$ global harmonic functions $v^1, \cdots, v^{k+l}$ on $X$, with  $v(p)=0$ and for  $x\in X$,
$$|\nabla v(x)|\leq Cd(x, p)^{\epsilon}+C,$$ 
$$-\kern-1em\int_{B_1(p)}\left<\nabla v^a, \nabla v^b\right>=\delta^{ab}, \forall\, 1\leq a, b\leq k+l, $$ 
 there are $k$ linear combinations of $v^{1}, \cdots, v^{k+l}$ which are linear with respect to the $\Bbb R$-factors in $X$  and  form a basis of linear functions on $\Bbb R^k$.
 \end{Prop}

 \section{Transformation theorems}
 In this section, we will give the proofs of Theorem~\ref{trans}. The idea is similar as in \cite{CJN, HH} by using the dimension gap Theorem~\ref{thm:gap}.
 We first establish the transformation theorem for a prescribed maximal
family. Then prove 
Theorem~\ref{trans} by constructing a maximal family in a limit on
expanding balls. At last, we will give an application to compare Euclidean factors of covering space and its base at two separated scale (see Theorem~\ref{thm:partial-transfer}).
\subsection{Transformation theorem}
\begin{Thm}\label{thm:maximal-family}
Given $n\geq2$, $\epsilon>0$, $\eta>0$, $1\leq k\leq m\leq n$,
and $\Phi$ as in Theorem~\ref{trans}, there are constants $\delta_0,\beta_0>0$,
depending only on $n,\epsilon,\eta,\Phi$,
with the following property. Let $(M^n,g,p)$ be complete,
$\op{Ric}_M\geq-(n-1)\delta g$, $0<\delta<\delta_0$, and $0<a<1$.
Suppose \eqref{non-dec} holds on $[a,2)$ for $j=k,\ldots,n$, and
$\theta_{p,k}(s)<\delta$ for every $a\leq s\leq2$.
If $m<n$, assume in addition the quantitative maximality condition
\begin{equation}
 \theta_{p,m+1}(1)\geq\eta. \label{gap}
\end{equation}
No additional maximality condition is needed when $m=n$.
For a $(\beta,m)$-splitting map $u:B_2(p)\to\R^m$,
$0<\beta<\beta_0$, $u(p)=0$, there is one constant rank-$k$ matrix
$A\in\R^{k\times m}$ such that $v=Au$ has the following property:
for every $a\leq s\leq1$, its lower triangular energy normalization
$T_s$, with positive diagonal, makes $T_sv$ an
$(\epsilon,k)$-splitting map on $B_s(p)$, with $G_s(T_sv)=I_k$ and the comparison estimate
\eqref{eq:target-matrices}.
\end{Thm}

\begin{proof}

By \cite[Lemma 4.2, 4.3]{HH}, we only need to find $T_s$ such that $T_s v$ is $(\epsilon, k)$-splitting and $G_s(T_sv)=I_k$. 

Let 
$ \Phi_*(t)=t+\sup_{0<s\leq t}\Phi(s)$.
If $m<n$, choose $\eta\geq \gamma>0$ such that
$\Phi_*(2\gamma)<\eta$. Assume $\delta<\gamma$. Then \eqref{non-dec} and \eqref{gap} implies that $\theta_{p, m+1}(s)\geq \gamma$ for $a\leq s\leq 1$. If $m=n$, there is $\xi_n>0$ such that $\theta_{p, n+1}(s)\geq \xi_n$ for $0<s\leq 1$ and in this case let $\gamma=\xi_n$ and let $\delta<1$. 

Note that if 
 $m=k$, the desired conclusion follows from
Theorem~\ref{thm:fixedrank} (\cite[Theorem 4.1]{HH}). 

In the following we always assume that $k<m$.
We will prove the assertion by induction on $m$. And by Theorem~\ref{thm:fixedrank}, the results hold for $m=1$.

Suppose the assertion holds for any number $< m$ and fails for $m$. There are $\delta_i,\beta_i\to0$,
complete pointed manifolds $(M_i,g_i,p_i)$, numbers $a_i\in(0,1)$,
and harmonic $(\beta_i,m)$-splitting maps
$u_i:(B_2(p_i),p_i)\to(\R^m,0)$ satisfying the hypotheses, for
which no fixed rank-$k$ selection has the required normalized
splitting property on all of $[a_i,1]$. And we may assume that $a_i\to 0$.

Take $\tau>0$ small which will be determined later. Let
\[
 E_i=\{r\in[a_i,1/2]:\theta_{p_i,m}(r)\leq\tau\},
 \qquad b_i=\min E_i.
\]
Since $\beta_i\to 0$,  by Theorem~\ref{split-map} (2), $b_i\to 0$ as $i\to \infty$. 

If for $i$ large, $b_i=a_i$, take $\tau$ such that $\Phi_{*}(\tau)<\delta_0(n, \epsilon, \gamma)$ as in Theorem~\ref{thm:fixedrank} (cf. \cite[Theorem 4.1]{HH}), the results hold.

If $b_i>a_i$, then 
\begin{align}
 \theta_{p_i,m}(b_i)&=\tau,
 \label{eq:transition-equality}\\
 \theta_{p_i,m}(r)&>\tau &&(a_i\leq r<b_i),
 \label{eq:transition-inner}\\
 \theta_{p_i,m}(r)&\leq\Phi_*(\tau) &&(b_i\leq r\leq1/2).
 \label{eq:transition-outer}
\end{align}
On $[b_i,1/2]$, by Theorem~\ref{thm:fixedrank}, for $\Phi_*(\tau)<\delta_0(n, \epsilon', \gamma)$, there are lower triangular
matrices $S_{i,t}$ with positive diagonal, satisfying
\begin{equation}\label{eq:outer-normalizations}
 G_t(S_{i,t}u_i)=I_m,
 \qquad S_{i,t}u_i\text{ is }(\epsilon',m)\text{-splitting},
 \qquad b_i\leq t\leq1/2.
\end{equation}

Rescale distances and measures by
\[
 (\widehat M_i,\hat p_i,\hat d_i,\hat\mu_i)
 =\left(M_i,p_i,b_i^{-1}d_i,
       \frac{\vol(\cdot)}{\vol(B_{b_i}(p_i))}\right),
 \qquad v_i=b_i^{-1}S_{i,b_i}u_i.
\]
As in the discussion of \cite[(4.9)]{HH}, by \eqref{eq:outer-normalizations},
\begin{equation}\label{eq:induction-growth-energy}
 -\kern-1em\int_{B_R(\hat p_i)}|\hat\nabla v_i|^2\,d\hat\mu_i
 \leq C(n)R^{C(n)\epsilon'},\qquad 1\leq R\leq(2b_i)^{-1}.
\end{equation}
The local mean-value inequality for harmonic gradients (\cite[Lemma 3.6]{MN}) applied on a
ball of twice the radius, yields, after increasing $C(n)$,
\begin{equation}\label{eq:induction-growth-gradient}
 |\hat\nabla v_i|(x)
 \leq C(n)(1+\hat d_i(x,\hat p_i))^{C(n)\epsilon'}
 \quad\hbox{on }B_{1/(8b_i)}(\hat p_i).
\end{equation}

Assume that
$$ (\widehat M_i,\hat p_i,\hat d_i,\hat\mu_i) \xrightarrow{\pmGH} (X, p, d, \mu)\in \op{RCD}(0, n).$$
Then
$X$ is $k$-splitting and $\theta_{p, m+1}(R)\geq \gamma$ for $R\geq 1$. And by \cite[Theorem 4.4]{AH}, $v_i$ is locally uniform and strong $W_{\op{loc}}^{1,2}$ convergence to a harmonic map $v:X\to \Bbb R^m$ with
\[
 v(p)=0,\qquad G_1(v)=I_m,\qquad
 |v(x)|\leq C(n)(1+d(x,p)^{1+C(n)\epsilon'}).
\]

Let $q$ be the largest dimension of a Euclidean factor of $X$.
Then $m>q\geq k$.  Write $X=\R^q\times Y$ with $Y$ containing no
line. For $\epsilon'$ small such that $C(n)\epsilon'<\epsilon(n, \gamma)$ and $\epsilon'<\sigma(n, 2, \gamma)$ as in Proposition~\ref{tech}, selecting all $q$ Euclidean
coordinates, there is a constant matrix $P\in\R^{q\times m}$ such that
\begin{equation}\label{eq:select-full-q}
 PP^{\mathsf T}=I_q,\qquad Pv=(x^1,\ldots,x^q).
\end{equation}
And thus $Pv_i$ is a
$(\epsilon_i(L),q)$-splitting map on $B_{2L}(\hat p_i)$, where
$\epsilon_i(L)\to0$ for every fixed $L<\infty$.

Since $q$ is the maximal Euclidean splitting dimension of $X$, there
are a finite $L\geq2$ and $\eta'>0$ such that
\begin{equation}\label{eq:new-maximality}
 \theta_{p,q+1}(L)>2\eta'.
\end{equation}
 Then
\[
 \theta_{p_i,q+1}(Lb_i)\geq\eta'
\]
for large $i$. On the distance scale $Lb_i$, consider
\[
 w_i=(Lb_i)^{-1}P S_{i,b_i}u_i:B_2\to\R^q,
 \qquad a_i'=\frac{a_i}{Lb_i}<1
\]
and spaces $\left(M_i, p_i, (Lb_i)^{-1}d_i, \frac{\volume}{\volume(B_{Lb_i}(p_i))}\right)$. Then
the error of this splitting map $w_i$ tends to zero. The $k$-Euclidean
control holds for every radius in $[a_i',2]$ with error $\delta_i$.
The original outward propagation inequality also holds there with
the same $\Phi$ and $\delta_i$, because restricting its upper endpoint
only decreases the supremum. The curvature error is at most
$\delta_i$ for large $i$, since $Lb_i<1$.
Thus all hypotheses of the induction statement for 
$q<m$ hold, with maximality gap $\eta'$.

Now by induction assumption, for $i$ large, there is a fixed rank-$k$ matrix
$D_i\in\R^{k\times q}$ such that the energy normalizations of
\begin{equation}\label{eq:final-nested-selection}
 f_i=A_i u_i,\qquad A_i=D_iP S_{i,b_i},
\end{equation} satisfying that for every $a_i\leq r\leq Lb_i$, there is lower triangular energy normalization $T^i_r$ such that
$T^i_rf_i$ is an $(\epsilon, k)$-splitting map on $B_r(p_i)$ with $G_r(T_r^i f_i)=I_k$. 

For $b_i\leq r\leq1/4$, write
\[
 f_i=(A_i S_{i,2r}^{-1})(S_{i,2r}u_i).
\]
The second factor is $(\epsilon',m)$-splitting on $B_{2r}$ by
\eqref{eq:outer-normalizations}, and the first factor is a constant
rank-$k$ matrix. Lemma~\ref{lem:uniform-selection} and $\delta_i\to0$ imply that the
energy normalization of $f_i$ on $B_r$ is $(\epsilon,k)$-splitting for $\epsilon'<\sigma(n, \epsilon)$ there.
For $1/4<r\leq1$, apply the same lemma directly to the original
$(\beta_i,m)$-splitting map $u_i$ on $B_2$ and the matrix $A_i$.
Consequently this one $A_i$ works at every radius in $[b_i,1]$.
Thus, it works on all of $[a_i,1]$.
On the overlap, both constructions give the same lower triangular
normalization, by uniqueness of Cholesky decomposition.
This contradicts the choice of the counterexample sequence.

\end{proof}

\begin{proof}[Proof of Theorem~\ref{trans}]
Fix $n$, $\epsilon>0$, and $\Phi$.  We argue by contradiction to the
existence of the pair $R_*,\delta_*$. If no such pair exists, there
are complete pointed manifolds $(M_i,g_i,p_i)$, parameters
\[
 R_i\to\infty,\qquad \delta_i\to0,
 \qquad 0<a_i<1,
\]
satisfying all the hypotheses on $[a_i,R_i]$, but admitting no map
with the asserted conclusion for error $\epsilon$. 

Assume that
\[
 \left(M_i,p_i,d_{g_i},\frac{\vol}{\vol(B_1(p_i)}\right)\xrightarrow{\pmGH}(X,p,d,\mu),
 \qquad (X,d,\mu)\in\RCD(0,n).
\]
For every $a_i\leq r<R_i$,
$\theta_{p_i,k}(r)<\delta_i\to0$. And thus
 $X$ splits $\R^k$.
Choose its maximal Euclidean factor and write
\[
 X=\R^m\times Z,\qquad k\leq m\leq n,
\]
where $Z$ contains no line. By the $\RCD$ splitting theorem (cf. \cite{Gi, BNS}), this
is also a measured product decomposition.

If $m<n$, there exist a finite $L\geq4$ and $\eta>0$ such that
\begin{equation}\label{eq:existence-limit-gap}
 \theta_{p,m+1}(L)>2\eta.
\end{equation}
Indeed, otherwise every ball of radius at least $4$ would have
zero distance to the $(m+1)$-product-ball class. Product compactness
on an exhaustion would then give a global $\R^{m+1}$ factor,
contrary to maximality. Fix $L$ and $\eta$ now. Convergence gives
\begin{equation}\label{eq:existence-sequence-gap}
 \theta_{p_i,m+1}(L)\geq\eta
\end{equation}
for all sufficiently large $i$. If $m=n$, take $L=4$ and
$\eta=1$.
The choices of $L$ and $\eta$ may depend on the limit $X$.

Pointed measured convergence to the measured product
$\R^m\times Z$ on $B_{8L}$ and
Theorem~\ref{split-map}, rescaled by $2L$, provide harmonic maps
\begin{equation}\label{eq:existence-full-family}
 u_i:B_{2L}(p_i)\to\R^m,
 \qquad u_i(p_i)=0,
\end{equation}
which are $(\beta_i,m)$-splitting, where $\beta_i\to0$.
The ball $B_{8L}$ is available since $R_i\to\infty$ and $L$ is
fixed. 

Put
\[
 \Delta_i=(1+L^2)\delta_i\to0,
 \qquad \bar g_i=L^{-2}g_i,
 \qquad \bar u_i=L^{-1}u_i.
\]
On the rescaled manifolds $(M_i, p_i, \bar g_i)$, $\bar u_i$ is a
$(\beta_i,m)$-splitting map on $B_2$, and
\[
 \Ric_{\bar g_i}\geq-(n-1)\Delta_i\bar g_i.
\]

Set
\[
 \widehat\Phi(t)
 :=t+\sup_{0<z\leq t}\min\{1,\Phi(z)\},\qquad t>0.
\]
Then $\widehat\Phi$ is finite, nondecreasing, and tends to zero as $t\to 0$. And $\hat \Theta_{p_i, j}(r)\leq \widehat\Phi(\max\{\delta, \theta_{p_i, j}(r)\})$ holds  on $[a_i/L,2)$ after enlarging its
parameter $\delta$ to $\Delta_i$. For every  $s\in[a_i/L,2]$ the rescaled ball is
$(\delta_i,k)$-Euclidean. When $m<n$,
\eqref{eq:existence-sequence-gap} is exactly the required gap at
rescaled radius one.

The parameters $L,\eta,\widehat\Phi$ are now fixed, whereas both
$\Delta_i$ and $\beta_i$ tend to zero. Thus
Theorem~\ref{thm:maximal-family}, with target error $\epsilon$, applies
for all large $i$. It gives a constant rank-$k$ matrix $A_i$.
Returning to the original metric, put
\[
 f_i=A_i u_i:B_{2L}(p_i)\to\R^k.
\]
For every $s\in[a_i,L]$, the lower triangular normalization
$N_{i,s}$ of $f_i$ satisfies $G_s(N_{i,s}f_i)=I_k$, and
$N_{i,s}f_i$ is $(\epsilon,k)$-splitting on $B_s(p_i)$.
The matrix comparison holds throughout $[a_i,L]$, since its
radius ratios are unchanged by rescaling.

Since $a_i<1<2<L$, define
\[
 v_i=N_{i,2}f_i\big|_{B_2(p_i)}.
\]
This map vanishes at $p_i$, is $(\epsilon,k)$-splitting on $B_2$, and
has $G_2(v_i)=I_k$. For $s\in[a_i,1]$ its energy normalization is
\[
 T_{i,s}=N_{i,s}N_{i,2}^{-1}.
\]
The splitting estimates and \eqref{eq:target-matrices} therefore
hold with $\epsilon$. This contradicts the choice
of the counterexamples and proves the existence of $R_*,\delta_*$.
\end{proof}

\subsection{Covering transfer}
In this subsection, we will compare Euclidean factors of covering space and its base at two separated scale. Recall that in \cite[Lemma 3.7]{H}, Huang derived that if for a non-negative Ricci curvature manifold $(M, p)$, its universal cover $(\tilde M, \tilde p)$ satisfies that $B_r(\tilde p)$ is $(\delta, k)$-Euclidean and not $(\eta, k+1)$-Euclidean for any $r\geq 1$ where $\delta$ depends on $\eta$, then the number of almost Euclidean factors of $B_r(p)$ is non-increasing respect to $r\geq 1$ in the Gromov-Hausdorff sense. Here we will show that

\begin{Thm}\label{thm:partial-transfer}
Given $n\geq2$, $\epsilon>0$, $\eta>0$, and a function
$\Phi:(0,\infty)\to(0,\infty)$ with $\Phi(t)\to0$ as $t\to 0$,
there are constants
\[
 \Lambda=\Lambda(n,\epsilon,\eta,\Phi)>100,\qquad
 \delta_0=\delta_0(n,\epsilon,\eta,\Phi)>0
\]
with the following property. Let
$\pi:(\widetilde M,\tilde p)\to(M^n,g,p)$ be the Riemannian universal
cover of a complete manifold with $\op{Ric}_{M^n}\geq0$. Fix $0<a<1$,
$0<\delta<\delta_0$, and integers $1\leq k,q\leq m\leq n$.
Suppose that:
\begin{enumerate}
\item for $a\leq t<\Lambda$ and $j=k,\ldots,n$,
\begin{equation}\label{eq:partial-propagation}
 \sup_{t\leq s<\Lambda}\theta_{\tilde p,j}(s)
 \leq\Phi\bigl(\max\{\delta,\theta_{\tilde p,j}(t)\}\bigr);
\end{equation}
\item $\theta_{\tilde p,k}(a)<\delta$;
\item if $m<n$, then $\theta_{\tilde p,m+1}(1)\geq\eta$;
\item a harmonic $(\delta,q)$-splitting map
$u:(B_\Lambda(p),p)\to(\R^q,0)$ is given.
\end{enumerate}
Set $h=\max\{0,q+k-m\}$. If $h>0$, there is a constant rank-$h$
matrix $A\in\R^{h\times q}$ such that $v=Au$, restricted to $B_2(p)$,
is $(\epsilon,h)$-splitting and $G_2(v)=I_h$. For every $r\in[a,1]$,
the lower triangular energy normalization $T_r$ of $v$, with positive
diagonal, makes $T_rv$ an $(\epsilon,h)$-splitting map on $B_r(p)$.
Moreover, for $s,t\in[a,1]$,
\begin{equation}\label{eq:partial-matrices}
 \|T_sT_t^{-1}\|\leq(1+C(n)\epsilon)
 \max\{(s/t)^{C(n)\epsilon},(t/s)^{C(n)\epsilon}\}.
\end{equation}
\end{Thm}

The integer $m$ is an upper bound enforced by the third hypothesis;
no $(\delta,m)$-Euclidean approximation upstairs is assumed. 

\begin{proof}
It suffices to consider $h>0$. 
Suppose no pair $\Lambda,\delta_0$ works. Choose $\Lambda_i\to\infty$
and counterexamples with $0<\delta_i<\Lambda_i^{-2n-4}$ and $0<a_i<1$.
Pass to a subsequence with fixed $k,q,m$, and write
$u_i:B_{\Lambda_i}(p_i)\to\R^q$ for the prescribed maps. In particular,
\begin{equation}\label{eq:expanding-transfer-errors}
 \delta_i\Lambda_i^{n+2}\to 0.
\end{equation}
Let $\tilde u_i=u_i\circ\pi_i$.
By
\cite[Lemma 5.3]{HH} and Bishop--Gromov comparison,
 for each fixed $S>0$, 
\[
  -\kern-1em\int_{B_S(\tilde p_i)}
 |\ip{\nabla\tilde u_i^a}{\nabla\tilde u_i^b}-\delta_{ab}|
 \leq C(n)\delta_i(\Lambda_i/S)^n\to 0.
\]
The corresponding scaled Hessian averages tend to zero as well, and
the pointwise gradient bounds hold on all these balls.

Assume that
$$\left(\tilde M_i, \tilde p_i, \tilde g_i, \frac{\volume(\cdot)}{\volume(B_1(\tilde p_i))}\right)\xrightarrow{\pmGH}(X, p, d, \mu).$$ 
Then $X$ splits $\Bbb R^k$. Write
\[
 X=\R^d\times Z,\qquad Z\text{ contains no line},
\]
with $m\geq d\geq k$. Then by Theorem~\ref{split-map}, there is $w_i: B_{l_i}(\tilde p_i)\to \Bbb R^d$ which is a $(\beta_i, d)$-splitting map where $l_i<\Lambda_i/10$, $l_i\to \infty$ and $\beta_il_i^{n+2}\to 0$. 

By \cite[Theorem 4.4]{AH}, assume $\tilde u_i\to v$, $w_i\to w$ locally uniformly and strongly in $W^{1,2}$.
Strong convergence preserves the Gram identities, so the limiting
components are harmonic and satisfy
\[
 \ip{\nabla v^a}{\nabla v^b}=\delta_{ab},\qquad
 \ip{\nabla w^c}{\nabla w^d}=\delta_{cd}.
\]
These identities hold on every bounded ball, hence globally.
Their constant gradient norms and the Bochner inequality on the
$\RCD(0,n)$ limit imply that all these Hessians vanish.

 We first show that there is a constant $q\times d$ matrix $C$ such that
\begin{equation}\label{eq:lift-limit-coordinates}
 v=Cw,\qquad CC^{\mathsf T}=I_q.
\end{equation}
Indeed, the Hessian product rule and the vanishing Hessians of $w$ and $v$ imply that each pairing
$\ip{\nabla v^a}{\nabla w^c}$ has zero weak differential.
Connectedness and the local Poincar\'e inequality on $X$ make it
a constant, denoted by $c_{ac}$. Set $C=(c_{ac})$ and
\[
 z^a=v^a-\sum_{c=1}^d c_{ac}w^c.
\]
The Gram identities give
\[
 \Delta z^a=0,\qquad
 \ip{\nabla z^a}{\nabla w^c}=0,\qquad
 |\nabla z^a|^2=1-\sum_{c=1}^d c_{ac}^2=:b_a^2.
\]
If $d<n$ and $b_a>0$, the global harmonic functions
\[
 w^1,\ldots,w^d,b_a^{-1}z^a
\]
have orthonormal gradients. Splitting theorem on $\op{RCD}(0, N)$-spaces (cf. \cite{Gi, BNS}) implies that $X$ would splits $\R^{d+1}$, contradicting maximality of $d$.
If $d=n$, the same theorem applied to $w$ identifies $X$ with
$\R^n$, and orthogonality to all its coordinate gradients already
forces $b_a=0$. Thus every $z^a$ is constant. Since $v$ and $w$
vanish at the base point, these constants are zero. This proves
$v=Cw$, and taking gradient inner products proves
$CC^{\mathsf T}=I_q$. 

Complete the rows of $C$ to an orthogonal matrix and write
\[
 O=\begin{pmatrix}C\\D\end{pmatrix}\in O(d),
 \qquad D\in\R^{(d-q)\times d}.
\]
Define on $B_{l_i}(\tilde p_i)$
\begin{equation}\label{eq:prescribed-completion}
 V_i=\begin{pmatrix}\tilde u_i\\Dw_i\end{pmatrix},
\end{equation}
with the last block omitted when $q=d$. Its first $q$ components
are exactly the original lifts for every $i$. Set
\[
 e_i=\tilde u_i-Cw_i,\qquad
 V_i=Ow_i+\begin{pmatrix}e_i\\0\end{pmatrix}.
\]
Then $v=Cw$ gives $e_i\to0$ locally uniformly, and
each component of $e_i$ is harmonic. For every fixed $S>0$,
Caccioppoli's inequality and Bishop--Gromov comparison yield
\[
 \mathcal E_i(S):=
-\kern-1em\int_{B_S(\tilde p_i)}|\nabla e_i|^2
 \leq\frac{C(n)}{S^2}
      -\kern-1em\int_{B_{2S}(\tilde p_i)}|e_i|^2
 \to 0.
\]
And
\[
 \gamma_i(S):=\sum_{c,d=1}^d
-\kern-1em\int_{B_S(\tilde p_i)}
 \left|\ip{\nabla w_i^c}{\nabla w_i^d}-\delta_{cd}\right|
 \to 0.
\]
Since $O$ is orthogonal and the gradients of $w_i$ are uniformly
bounded, expansion of the Gram entries and Cauchy--Schwarz give
\[
-\kern-1em\int_{B_S(\tilde p_i)}
 \left|\ip{\nabla V_i^\alpha}{\nabla V_i^\beta}
                 -\delta_{\alpha\beta}\right|
 \leq C(n)\bigl(\gamma_i(S)+\mathcal E_i(S)^{1/2}
                              +\mathcal E_i(S)\bigr)
 \to 0.
\]
Each component of $V_i$ is harmonic, vanishes at the base point,
and has a uniform gradient bound by its defining formula
\eqref{eq:prescribed-completion}. Then $V_i$ is
$(\epsilon_i(S),d)$-splitting on $B_{2S}(\tilde p_i)$, with
$\epsilon_i(S)\to0$.

If $d<n$, maximality and product compactness provide a fixed $L\geq64$
and $\eta'>0$ such that $\theta_{x,d+1}(L)>2\eta'$. Otherwise exact
$(d+1)$-splittings on an exhausting sequence would give a global one.
For large $i$, $\theta_{\tilde p_i,d+1}(L)\geq\eta'$.
If $d=n$, take $L=64$ and omit this gap. 

Put $\Delta_i=2(\delta_i+\widehat\Phi(\delta_i))\to0$.
Propagation from $a_i$ gives
\begin{equation}\label{eq:transfer-allscale-k}
 \theta_{\tilde p_i,k}(t)<\Delta_i\qquad(a_i\leq t\leq2L),
\end{equation}
and propagation holds there with parameter $\Delta_i$ and modulus
$\widehat\Phi$. Rescale by $L^{-1}$ and apply
Theorem~\ref{thm:maximal-family} to $L^{-1}V_i$, with full dimension $d$
and lower endpoint $a_i/L$.  For each fixed sufficiently small $e>0$, all large $i$
admit one rank-$k$ matrix $P_i$ such that the energy normalizations
of $P_iV_i$ are $(e,k)$-splitting at every radius in $[a_i,L]$,
in the original scale.

On the common domain $B_{2L}(\tilde p_i)$ set
\[
 E_i=\Span\{V_i^1,\ldots,V_i^d\},\quad
 U_i=\Span\{\tilde u_i^1,\ldots,\tilde u_i^q\},\quad
 F_i=\Span\{(P_iV_i)^1,\ldots,(P_iV_i)^k\}.
\]
Since $U_i,F_i\subset E_i$,
\begin{equation}\label{eq:intersection}
 \dim(U_i\cap F_i)\geq q+k-d\geq h.
\end{equation}
Choose $h$ independent functions in this intersection. As a vector,
they have both representations
\begin{equation}\label{eq:intersection-representations}
 z_i=B_i\tilde u_i=D_iP_iV_i,
\end{equation}
where $B_i$ and $D_i$ have rank $h$. In particular $z_i$ is the
lift of the base harmonic map $B_i u_i$. 

 For large $i$, and $a_i\leq \rho_i$ with
$24\rho_i\leq L$. Normalize $P_iV_i$ on $B_{24\rho_i}$ by $S_{i, 24\rho_i}$. 
By Lemma~\ref{lem:uniform-selection}, for $(D_i, S_{i, 24\rho_i}^{-1}) S_{i, 24\rho_i}P_iV_i$, the normalize of it by $K_i$ is
 is $(\tau, h)$-splitting map on $B_{12\rho_i}$. 

Let
\[
 H_i=K_{i, \rho_i}B_i u_i\quad\hbox{on the base}.
\]
Apply \cite[Lemma 5.3]{HH} to 
\[
 \left|\ip{\nabla H_i^a}{\nabla H_i^b}-\delta_{ab}\right|
 \quad\hbox{and}\quad |\Hess H_i^a|^2.
\]
Their lifts have the splitting bounds on $B_{12\rho_i}$.
Consequently,
$$
-\kern-1em\int_{B_{4\rho_i}(p_i)}
 \left|\ip{\nabla H_i^a}{\nabla H_i^b}-\delta_{ab}\right|
 \leq C(n)\tau$$
 $$
 (4\rho_i)^2-\kern-1em\int_{B_{4\rho_i}(p_i)}|\Hess H_i^a|^2
 \leq \frac19 C(n)\tau.
$$
It follows that 
$H_i$ is $(C(n)\tau,h)$-splitting on $B_{4\rho_i}(p_i)$.

Restriction to $B_{\rho_i}$ changes these errors by at most a dimensional
factor. The Gram matrix is $I_h+O_n(\tau)$, so a further lower
triangular matrix $N_{i,\rho_i}=I_h+O_n(\tau)$ normalizes it on $B_{\rho_i}$.
 For sufficiently small $\tau$,
\begin{equation*}
 N_{i,\rho_i}H_i\text{ is }(\epsilon,h)\text{-splitting on }B_{\rho_i}(p_i)
 \quad\text{for every }\rho_i\in[a_i,2].
 \end{equation*}
 
 Set $A_i=N_{i,2}K_{i, 2}B_i$ and $v_i=A_i u_i$. Then $G_2(v_i)=I_h$ and
$v_i$ is $(\epsilon,h)$-splitting on $B_2$. Its energy normalization at
$r\in[a_i,1]$ is
\[
 T_{i,r}=N_{i,r}K_{i, r}K_{i, 2}^{-1}N_{i,2}^{-1},
\]
\cite[Lemma 4.2, 4.3]{HH} give the matrix estimate, which
contradicts the counterexamples.
\end{proof}

A limit case of above result is:

\begin{Cor}\label{split-2}
Fix $n$ and a function $\Phi:(0,\infty)\to(0,\infty)$ with
$\Phi(t)\to0$ as $t\to 0$.
Let $\pi_i:(\widetilde M_i,\tilde p_i)\to(M_i,p_i)$ be Riemannian
universal covers of complete $n$-manifolds with $\Ric\geq0$.
Suppose $0<r_i<R_i$, $r_i/R_i\to0$, $\delta_i>0$, $\delta_i\to0$, and
$\Lambda_i\to\infty$. Assume that, for $j=k,\ldots,n$,
\begin{equation}\label{eq:transfer-propagation}
 \sup_{t\leq s<\Lambda_iR_i}\theta_{\tilde p_i,j}(s)
 \leq\Phi\bigl(\max\{\delta_i,\theta_{\tilde p_i,j}(t)\}\bigr),
 \qquad r_i\leq t<\Lambda_iR_i.
\end{equation}
Assume the following pointed GH convergences:
\begin{align}
 (\widetilde M_i,R_i^{-2}\tilde g_i,\tilde p_i)
 &\to (X_+,x_+)=(\R^m\times Z_+,(0,z_+)),
 \label{eq:transfer-upper-limit}\\
 (M_i,R_i^{-2}g_i,p_i)
 &\to (Y_+,y_+)=(\R^q\times W_+,(0,w_+)),
 \label{eq:transfer-base-limit}\\
 (\widetilde M_i,r_i^{-2}\tilde g_i,\tilde p_i)
 &\to (X_-,x_-)=(\R^k\times Z_-,(0,z_-)),
 \label{eq:transfer-lower-limit}
\end{align}
where $Z_+$ contains no line, $0\leq k\leq m\leq n$, and
$0\leq q\leq m$. Set $h=\max\{0,q+k-m\}$.
Then, for every fixed $s>0$,
\begin{equation}\label{eq:transfer-conclusion}
 \theta_{p_i,h}(sr_i)\to0.
\end{equation}
In particular, every pointed GH limit of
$(M_i,r_i^{-2}g_i,p_i)$ splits $\R^h$.
No maximality assumption is needed on $Z_-$ or $W_+$.
\end{Cor}

\section{A counterexample with a compact flat factor}\label{sec:example}

We now prove Theorem~\ref{main-2}. This section is independent of
the transformation theorem and its rank induction. The rank is $1$, so all transformation
matrices are positive scalars. The manifold we construct here is a product of two part: one almost splits at infinity and does not split itself by using \cite{CN}'s construction; one has splitting structure in local. 

\subsection{A first factor with different behavior at two scales}

\begin{Lem}[Geometric building block]\label{lem:geometry}
For every $\alpha>0$ there is a complete smooth pointed three-manifold
$(P,g_P,p)$ with $\op{Ric}_{P}\geq0$ such that
\begin{equation}\label{eq:badball}
 \op{diam}\overline B_1^P(p)<\frac32,
\end{equation}
and, for some $R_0\geq10$, every $B_r^P(p)$ with $r\geq R_0$ is
$(\alpha,1)$-Euclidean. The manifold can be chosen to have a unique
asymptotic cone, as close as desired in pointed measured
Gromov--Hausdorff topology to
\begin{equation}\label{eq:targetcone}
 C_*:=\R\times C\left(\mathbb{S}^1(1/4)\right).
\end{equation}
\end{Lem}

\begin{proof}

We use the smooth, equal-area family of metrics on $\mathbb{S}^2$ from
\cite[Section 3]{CN}. It joins the round metric
$\gamma_-=\tfrac14g_{\mathbb{S}^2(1)}$ to smooth metrics $\gamma_+$ arbitrarily
close to the spherical suspension of $\mathbb{S}^1(1/4)$, with
$K_{\gamma}\geq1$ and area $\pi$. Indeed, the underlying two-parameter
models are
\[
 dt^2+b^2\sin^2(t/a)\,d\theta^2,\qquad
 0<t<a\pi,\qquad ab=\tfrac14,
\]
with $a$ running from $1/2$ towards $1$, followed by smoothing the poles
and restoring the area as in that construction.

Choose a finite smooth path and extend it constantly at both ends.
After applying Moser's volume-form identification and slowing the
parameter, it satisfies the hypotheses of \cite[Lemma 2.1]{CN}.
That lemma gives a metric
\begin{equation}\label{eq:CNmetric}
 g_0=dr^2+r^2h(r)^2\gamma(f(r))
 \quad\hbox{on }(0,\infty)\times\mathbb{S}^2,
 \qquad \Ric_{g_0}\geq0,
\end{equation}
which is exactly $dr^2+r^2\gamma_-$ near zero and has the unique
asymptotic cone $C(\mathbb{S}^2,h_\infty^2\gamma_+)$.
The constant $h_\infty<1$ can be chosen arbitrarily close to $1$.

We spell out how to remove the vertex; no boundary remains in $P$.
Let $c=1/2$ and choose a smooth concave function $F$ satisfying
\[
 F(t)=t\ \hbox{near }0,\qquad c\leq F'(t)\leq1,
 \qquad F(t)=ct+b\ \hbox{for all sufficiently large }t,
\]
where $b>0$. The rotational metric
$dt^2+F(t)^2g_{\mathbb{S}^2(1)}$ is smooth at its center and has
nonnegative Ricci curvature: its radial Ricci curvature is
$-2F''/F$, and its tangential Ricci eigenvalue is
$(1-(F')^2-FF'')/F^2$.
Where $F=ct+b$, the change of variable $r=t+b/c$ identifies it with
$dr^2+c^2r^2g_{\mathbb{S}^2(1)}$.
A sufficiently small scaled copy therefore replaces the vertex in
\eqref{eq:CNmetric} smoothly, preserving $\Ric\geq0$.
Take $p$ to be the center of this cap.

Fix an observation radius well inside the region where
$g_0=dr^2+r^2\gamma_-$. As the cap size tends to zero, the closed ball
at this radius, rescaled to radius one, converges to the vertex-centered
unit ball of $C(\mathbb{S}^2(1/2))$. Possible shortcuts outside the conical
region are excluded by choosing that region several times larger than
the observation ball. The limiting ball has diameter $\sqrt2$, by the
cone distance formula, since the link has diameter $\pi/2$.
Choosing the cap small and then rescaling gives \eqref{eq:badball}.

The cap and this fixed rescaling do not affect the asymptotic cone.
For the metric \eqref{eq:CNmetric}, the limits of $h$ and the fact that
$\gamma$ is eventually constant give convergence of every blowdown to
$C(\mathbb{S}^2,h_\infty^2\gamma_+)$. Taking $\gamma_+$ close to the suspension
and $h_\infty$ close to $1$ makes its unit ball arbitrarily close to that
of $C_*$, which splits an $\R$-factor. Consequently all sufficiently
large balls, rather than merely a subsequence, are
$(\alpha,1)$-Euclidean.

The convergence can also be taken in the measured sense. The smooth
links have fixed area and a common lower curvature bound and converge
noncollapsed to the suspension; their area measures converge as well.
The cone measures are the corresponding angular measures times
$r^2\,dr$. In the blowdown of \eqref{eq:CNmetric}, this also follows
directly from its volume element; the bounded cap disappears.
\end{proof}

\begin{Rem}[A uniform obstruction at radius one]\label{rem:gap}
Every unit ball centered in $\R\times Z$ has diameter $2$, because it
contains $(1,z)$ and $(-1,z)$. Since
$|\op{diam} A-\op{diam} B|\leq2d_{\mathrm{GH}}(A,B)$,
\eqref{eq:badball} implies
\begin{equation}\label{eq:gap}
 d_{\mathrm{GH}}\bigl(\overline B_1^P(p),
 \overline B_1^{\R\times Z}((0,z))\bigr)>\frac14
 \quad\hbox{for every }(Z,z).
\end{equation}
This gap is independent of $\alpha$. By almost splitting
theorem \cite{Gi}, there is a fixed $\tau>0$ such that no
$(\tau,1)$-splitting function exists on $B_4^P(p)$.
\end{Rem}

\subsection{The flat compact factor and all-scale geometry}

Put $m=n-3$, and choose $P$ using Lemma~\ref{lem:geometry}.
Let
\begin{equation}\label{eq:product}
 F_A=\R^m/(A\mathbb Z)^m,\qquad
 Q=P\times F_A,\qquad q=(p,0),\qquad A>8R_0.
\end{equation}
Equip $Q$ with the product metric $g_Q$. The factor $F_A$ is complete,
compact, and flat, and $\Ric_{Q}\geq0$.

If $0<r\leq R_0$, the radius-$r$ ball in $F_A$, with its ambient metric,
is exactly a Euclidean ball.
It follows that
\begin{equation}\label{eq:smallgeom}
 B_r^Q(q)\cong B_r^{P\times\R^m}((p,0)).
\end{equation}
These balls are exactly $(0,m)$-Euclidean.

If $r\geq R_0$, use the $(\alpha,1)$-Euclidean approximation for
$B_r^P(p)$ and take its product with the identity on $F_A$.
This gives
\begin{equation}\label{eq:largegeom}
 d_{\mathrm{pGH}}\left(\overline B_r^Q(q),
 \overline B_r^{\R\times Z_r\times F_A}((0,z_r,0))\right)
 \leq C\alpha r.
\end{equation}
Choose $\alpha$ so that $C\alpha<\delta$. Equations
\eqref{eq:smallgeom}--\eqref{eq:largegeom} prove that
\begin{equation}\label{eq:allgeom}
 B_r^Q(q)\ \hbox{is }(\delta,1)\hbox{-Euclidean for every }r>0.
\end{equation}
Choosing $P$ such that its asymptotic cone is sufficient close to $C_{*}$. And then fixed $A$, thus $Q$ has the same measured asymptotic cone as $P$.  By Theorem~\ref{split-map}, there is $R_{\op{init}}$ such that for every $R\geq R_{\op{init}}$, $B_{2R}(q)$ admits a $(\delta, 1)$-splitting map $v_R$.

\subsection{Excluding every candidate splitting function}

Keep $P$, $A$, and hence the complete product $Q$ fixed. We claim that,
for a sufficiently small dimensional $\epsilon_*>0$, there is a finite
$R_*$ such that no harmonic function on $B_{2R}^Q(q)$ admits
$(\epsilon_*,1)$-splitting normalizations at all radii $0<r\leq R$, whenever
$R\geq R_*$.

Suppose otherwise. There are $R_j\to\infty$ and such harmonic
functions $u_j$. Normalize them by
\[
 u_j(q)=0,\qquad E_j(1)=1,\qquad
 E_j(r):= -\kern-1em\int_{B_r^Q(q)}|\nabla u_j|^2.
\]

By \cite[Lemmas 4.2--4.3 and Remark 4.4]{HH}, 
$T_{j,r}u_j:=E_j(r)^{-1/2}u_j$
is $(C_0\epsilon_*,1)$-splitting, with an absolute constant $C_0$,
and after rescaling $R_j$ to one, 
\[
 \max\{T_{j,r},T_{j,r}^{-1}\}\leq Cr^{c\epsilon_*},
 \qquad 1\leq r\leq R_j,
\]
since $T_{j,1}=1$.  The constants $C,c$ depend only on $n$, independently
of $j$ and $R_j$. Using the gradient bound for $T_{j,r}u_j$ and then
integrating along minimizing geodesics from $q$ yields
\begin{equation}\label{eq:polybound}
 \sup_{B_r^Q(q)}|\nabla u_j|\leq Cr^{c\epsilon_*},\qquad
 \sup_{B_r^Q(q)}|u_j|\leq Cr^{1+c\epsilon_*},\qquad 1\leq r\leq R_j.
\end{equation}
Choose $\epsilon_*$ small enough that $1+c\epsilon_*<2$.

Interior elliptic estimates on the fixed smooth manifold $Q$, followed
by a diagonal subsequence, now give
\begin{equation}\label{eq:limit}
 u_j\to u\quad\hbox{in }C^\infty_{\mathrm{loc}}(Q),
 \qquad\Delta_Qu=0,
 \qquad  -\kern-1em\int_{B_1^Q(q)}|\nabla u|^2=1.
\end{equation}
Thus $u$ is nonconstant and has subquadratic growth by
\eqref{eq:polybound}. For each fixed radius, its energy is positive,
and the energy normalizations above pass to the limit. Suitable positive
multiples of $u$ are therefore $(2C_0\epsilon_*,1)$-splitting on every
fixed ball.

Next lift $u$ to $\widetilde u(x,y)=u(x,[y])$ on
$P\times\R^m$. This covering is a
local isometry and does not increase distances, so $\widetilde u$ is
global harmonic and still has growth at most $1+c\epsilon_*<2$.
By \cite[Corollary 3.4]{HH}, its derivatives in the Euclidean directions
are constants. Integrating these derivatives gives
\[
 \widetilde u(x,y)=f(x)+\sum_{\ell=1}^m b_\ell y_\ell.
\]
The deck periodicity
$\widetilde u(x,y+Ae_\ell)=\widetilde u(x,y)$ forces $Ab_\ell=0$
for every $\ell$. Hence $u(x,[y])=f(x)$, and $f$ is harmonic on $P$.

We now transfer its splitting estimates from $B_8^Q(q)$ to
$B_4^P(p)$.
Let $w=bf$ be the normalization supplied at radius $8$ in $Q$.
Because $A>8R_0\geq80$, the relevant torus balls are Euclidean.
For every nonnegative function $H$ depending only on $x\in P$,
\begin{equation}\label{eq:fubini}
 \int_{B_8^Q(q)} H(x)\,d\vol_Q
 =\omega_m\int_{B_8^P(p)}
 (64-d_P(p,x)^2)^{m/2}H(x)\,d\vol_P.
\end{equation}
The weight is at least $\omega_m48^{m/2}$ on $B_4^P(p)$.
Also $\vol_P(B_8)/\vol_P(B_4)\leq2^3$. Therefore
\begin{equation}\label{eq:descent}
  -\kern-1em\int_{B_4^P(p)}H
 \leq D_n -\kern-1em\int_{B_8^Q(q)}H,
 \qquad D_n=8(4/3)^{m/2}.
\end{equation}
Apply this to $H=\bigl||\nabla w|^2-1\bigr|$ and
$H=|\Hess w|^2$. For a function independent of the torus variable,
its gradient and Hessian norms in $P$ and $Q$ agree. The pointwise
gradient bound descends directly, and the Hessian scaling changes
from $8^2$ to $4^2$. Thus $w$ is a $(C_n\epsilon_*,1)$-splitting function
on $B_4^P(p)$, for a dimensional constant $C_n$.
Choose $\epsilon_*$ small enough for the normalization and comparison
estimates above, for $1+c\epsilon_*<2$, and for $C_n\epsilon_*<\tau$,
with $\tau$ from Remark~\ref{rem:gap}.
This contradicts \eqref{eq:gap} and proves the claim.
The choice of $\epsilon_*$ is independent of $\delta$, $P$, $A$, and $R_*$.

\subsection{Unit-scale normalization and the initial splitting map}

Choose $R\geq\max\{R_*,R_{\mathrm{init}},1\}$ and set
\begin{equation}\label{eq:finalexample}
 (M,g,q):=(Q,R^{-2}g_Q,q).
\end{equation}
This is complete and smooth, with $\Ric_M\geq0$.
The scale invariance of \eqref{eq:allgeom} gives assertion (1) of
Theorem~\ref{main-2}, while $R^{-1}v_R$ gives assertion (2) by
Theorem~\ref{split-map}. If assertion (3) failed for a
harmonic function $u$ on $B_2^g(q)$, then $Ru$ on $B_{2R}^Q(q)$ would
admit $(\epsilon_*,1)$-splitting normalizations at every radius
$0<r\leq R$. This contradicts
the preceding claim and completes the proof.

\end{document}